\documentclass[11pt,twoside]{amsart}
\usepackage{amsxtra}
\usepackage{color}
\usepackage{amsopn}
\usepackage{amsmath,amsthm,amssymb,arydshln}
\usepackage{mathrsfs,mathtools}
\usepackage{stmaryrd}
\usepackage{hyperref}
\usepackage{enumerate}
\usepackage{xcolor}
\usepackage{pifont}
\usepackage{adjustbox}
\usepackage{tikz}

\newtheorem{theorem}{Theorem}[section]
\newtheorem{corollary}[theorem]{Corollary}
\newtheorem{proposition}[theorem]{Proposition}
\newtheorem{lemma}[theorem]{Lemma}
\newtheorem*{theorem*}{Theorem}
\theoremstyle{definition}

\newtheorem{example}[theorem]{Example}

\def\sideremark#1{\ifvmode\leavevmode\fi\vadjust{
		\vbox to0pt{\hbox to 0pt{\hskip\hsize\hskip1em
				\vbox{\hsize3cm\tiny\raggedright\pretolerance10000
					\noindent #1\hfill}\hss}\vbox to8pt{\vfil}\vss}}}

\newcommand{\R}{{\mathbb R}}

\newcommand{\Z}{\mathbb Z}
\newcommand{\T}{{\mathbb T}}

\newcommand{\beq}{\begin{equation}}
\newcommand{\eeq}{\end{equation}}

\newcommand{\g}{\gamma}

\renewcommand{\l}{\lambda}

\renewcommand{\o}{\omega}

\newcommand{\psip}{\psi}
\newcommand{\psim}{\widehat{\psi}}

\newcommand{\SU}{{\mathrm{SU}}}

\newcommand{\SO}{{\mathrm {SO}}}

\newcommand{\G}{{\mathrm G}}
\newcommand{\K}{{\mathrm K}}

\newcommand{\W}{\wedge}
\newcommand{\Diff}{\mathrm{Dif{}f}}
\newcommand{\Iso}{\mathrm{Iso}}

\DeclareMathOperator\End{End}
\DeclareMathOperator\Aut{Aut}

\newcommand{\frg}{\mathfrak{g}}

\newcommand{\frk}{\mathfrak{k}}

\newcommand{\frso}{\mathfrak{so}}
\newcommand{\frsu}{\mathfrak{su}}

\renewcommand{\gg}{\mathfrak{g}}

\newcommand{\st}{\ |\ }

\newcommand{\sst}{\scriptscriptstyle}

\numberwithin{equation}{section}

\title[On the automorphism group of a symplectic half-flat 6-manifold]{On the automorphism group of a symplectic half-flat 6-manifold}
\author{Fabio Podest\`a and Alberto Raffero}
\thanks{The authors were supported by GNSAGA of INdAM}
\address{Dipartimento di Matematica e Informatica ``U.~Dini'' \\ Universit\`a degli Studi di Firenze\\ Viale Morgagni 67/a\\ 50134 Firenze\\ Italy}
\email{podesta@math.unifi.it, alberto.raffero@unifi.it}

\begin{document}
\begin{abstract}
We prove that the automorphism group of a compact 6-manifold $M$ endowed with a symplectic half-flat SU(3)-structure has abelian Lie algebra 
with dimension bounded by $\mathrm{min}\{5,b_1(M)\}$. Moreover, we study the properties of the automorphism group action and we discuss relevant examples. 
In particular, we provide new complete examples on $T\mathbb{S}^3$ which are invariant under a cohomogeneity one action of $\SO(4)$.
\end{abstract}

\maketitle

\section{Introduction}
An $\SU(3)$-structure on a six-dimensional smooth manifold $M$ is the data of an almost Hermitian structure $(g,J)$ with fundamental 2-form $\omega\coloneqq g(J\cdot,\cdot)$ 
and a complex volume form $\Psi = \psi+ i\,\psim \in\Omega^{3,0}(M)$ such that
\begin{equation}\label{normcond}
\psi\W\psim = \frac23\,\omega^3.
\end{equation}

By \cite{Hit}, the whole data $(g,J,\Psi)$ is completely determined by the real 2-form $\omega$ and the real 3-form $\psi$, provided that they satisfy suitable conditions 
(see $\S$\ref{sect3} for more details). 

An $\SU(3)$-structure $(\omega,\psi)$ is said to be {\em symplectic half-flat} if both $\omega$ and $\psi$ are closed. 
In this case, the intrinsic torsion can be identified with a unique real $(1,1)$-form $\sigma$ which is primitive with respect to $\omega$, i.e., 
$\sigma\W\omega^2=0$, and fulfills $d\psim=\sigma\W\omega$ (see e.g.~\cite{ChSa}).  
This $\SU(3)$-structure is {\em half-flat} according to \cite[Def.~4.1]{ChSa}, namely $d(\omega^2)=0$ and $d\psi=0$, 
and the corresponding almost complex structure $J$ is integrable if and only if $\sigma$ vanishes identically. 
When this happens, $(M,\omega,\psi)$ is a {\em Calabi-Yau} 3-fold. Otherwise, the symplectic half-flat structure is said to be {\em strict}. 

In recent years, symplectic half-flat structures turned out to be of interest in supersymmetric string theory. 
For instance, in \cite{FiUg} the authors proved that supersymmetric flux vacua with constant intermediate $\SU(2)$-structure \cite{And} 
are related to the existence of special classes of half-flat structures on the internal 6-manifold. 
In particular, they showed that solutions of Type IIA SUSY equations always admit a symplectic half-flat structure.  
In \cite{LTY}, the definition of symplectic half-flat structures, which are called supersymmetric of Type IIA, is generalized in higher dimensions,  
and it is proved that semi-flat supersymmetric structures of Type IIA correspond to semi-flat supersymmetric structures of Type IIB via the SYZ and Fourier-Mukai 
transformations.

In mathematical literature, symplectic half-flat structures were first introduced and studied in \cite{DeB} and then in \cite{DeBTom0}, 
while explicit examples were exhibited in \cite{CoTo,DeBTom,FMOU,PoRa,ToVe}. 
Most of them consist of simply connected solvable Lie groups endowed with a left-invariant symplectic half-flat structure. 
Moreover, in \cite{FMOU} it was proved that every six-dimensional compact solvmanifold with an invariant symplectic half-flat structure also admits a solution of  
Type IIA SUSY equations.  

Let $M$ be a 6-manifold endowed with a strict symplectic half-flat structure $(\omega,\psi)$. 
In the present paper, we are interested in studying the properties of the automorphism group 
$\Aut(M,\omega,\psi) \coloneqq \left\{f\in\Diff(M)\st f^*\omega=\omega,~f^*\psi=\psi \right\}$, 
aiming at understanding how to construct non-trivial examples with high degree of symmetry. 

In \cite{PoRa}, we proved the non-existence of compact homogeneous examples and we classified all non-compact cases which are homogeneous under the action 
of a semisimple Lie group of automorphisms. 
Here, in Theorem \ref{MainThm} we show that the Lie algebra of $\Aut(M,\omega,\psi)$ is abelian with dimension bounded by $\mathrm{min}\{5,b_1(M)\}$ 
whenever $M$ is compact. 
This allows to obtain a direct proof of the aforementioned non-existence result. 
In the same theorem, we also provide useful information on geometric properties of the $\Aut^{\sst0}(M,\omega,\psi)$-action on the manifold, 
proving in particular that the automorphism group acts by cohomogeneity one only when $M$ is diffeomorphic to a torus.  
Some relevant examples are then discussed in order to show that the automorphism group can be non-trivial and that the upper bound on its dimension can be actually attained. 

As our previous result on non-compact homogeneous spaces suggests, the non-compact ambient might provide a natural setting where looking for new examples. 
In section \ref{sect3}, we obtain new complete examples of symplectic half-flat structures on the tangent bundle $T\mathbb{S}^3$ which are invariant 
under the natural cohomogeneity one action of $\SO(4)$. These include also the well-known Calabi-Yau example constructed by Stenzel \cite{Ste}.

\section{The automorphism group}
Let $M$ be a six-dimensional manifold endowed with an $\SU(3)$-structure $(\omega,\psi)$. 
The {\em automorphism group} of $(M,\omega,\psi)$ consists of the diffeomorphisms of $M$ preserving the $\SU(3)$-structure, namely
\[
\Aut(M,\omega,\psi) \coloneqq \left\{f\in\Diff(M)\st f^*\omega=\omega,~f^*\psi=\psi \right\}.
\]
Clearly, $\Aut(M,\omega,\psi)$ is a closed Lie subgroup of the isometry group $\Iso(M,g)$, 
as every automorphism preserves the Riemannian metric $g$ induced by the pair $(\omega,\psip)$. 
The Lie algebra of the identity component $\G \coloneqq \Aut^{\sst0}(M,\omega,\psi)$ is 
\[
\gg=\left\{X\in\mathfrak{X}(M) \st \mathcal{L}_{\sst X}\omega=0,~\mathcal{L}_{\sst X}\psi=0 \right\}, 
\]
and every $X\in\gg$ is a Killing vector field for the metric $g$. 
Moreover, the Lie group $\Aut(M,\omega,\psi)\subset\Iso(M,g)$ is compact whenever $M$ is compact. 

If $(M,\omega,\psi)$ is a Calabi-Yau 3-fold, i.e., if $\omega$, $\psi$ and $\psim$ are all closed, 
then the Riemannian metric $g$ is Ricci-flat and $\mathrm{Hol}(g)\subseteq\SU(3)$. 
When $M$ is compact and the holonomy group is precisely $\SU(3)$, it follows from Bochner's Theorem that $\Aut(M,\omega,\psi)$ is finite. 

We now focus on {\em strict symplectic half-flat} structures, namely $\SU(3)$-structures $(\omega,\psi)$ such that 
\[
d\omega=0,\quad d\psi=0,\quad d\psim=\sigma\W\omega,
\]
with $\sigma\in\left[\Omega^{\sst 1,1}_{\sst0}(M) \right]\coloneqq\left\{\kappa\in\Omega^{\sst2}(M) \st J\kappa=\kappa,~\kappa\W\omega^2=0 \right\}$ 
not identically vanishing.   
Notice that the condition on $\sigma$ is equivalent to requiring that the almost complex structure $J$ induced by $(\omega,\psi)$ is non-integrable (cf.~e.g.~\cite{ChSa}). 
In this case, we can show the following result. 

\begin{theorem}\label{MainThm}
Let $M$ be a compact six-dimensional manifold endowed with a strict symplectic half-flat structure $(\omega,\psi)$. 
Then, there exists an injective map
\[
\mathscr{F}:\gg\rightarrow\mathscr{H}^1(M),\quad X\mapsto \iota_{\sst X}\omega,
\]
where $\mathscr{H}^1(M)$ is the space of $\Delta_g$-harmonic 1-forms.  
Consequently, the following properties hold:
\begin{enumerate}[1)]
\item\label{thm1} $\dim(\gg)\leq b_1(M)$;
\item\label{thm2} $\gg$ is abelian with $\dim(\gg)\leq5$;
\item\label{thm3} for every $p\in M,$ the isotropy subalgebra $\gg_p$ has dimension $\dim(\gg_p)\leq2$. If $\dim(\gg_p)=2$ for some $p$, then $\G_p=\G$;
\item\label{thm4} the $\G$-action is free when $\dim(\gg)\geq4$. In particular, when $\dim(\gg) = 5$ the manifold $M$ is diffeomorphic to $\mathbb T^6$.
\end{enumerate}
\end{theorem}

Before proving the theorem, we show a general lemma. 
\begin{lemma}\label{prellem}
Let $(\omega,\psi)$ be an $\SU(3)$-structure. Then, for every vector field $X$ the following identity holds
\[
\iota_{\sst X}\psi\W\psi = -2*(\iota_{\sst X}\omega),
\]
where $*$ denotes the Hodge operator determined by the Riemannian metric $g$ and the orientation $dV_g=\frac{1}{6}\omega^3$. 
\end{lemma}
\begin{proof}
From the equation $\iota_{\sst X}\Psi\W\Psi=0$, which holds for every vector field $X$, we have
\[
\iota_{\sst X}\psi\W\psi = \iota_{\sst X}\psim\W\psim,\quad \iota_{\sst X}\psi\W\psim = -\iota_{\sst X}\psim\W\psi.
\]
Using the above identities and the relations $\iota_{\sst X}\psi = \iota_{\sst JX}\psim$, $\iota_{\sst JX}\psi = -\iota_{\sst X}\psim$, we get
\begin{eqnarray*}
\iota_{\sst X}\psi\W\psi	&=&	\iota_{\sst JX}\psim\W\psi\\
					&=&	\iota_{\sst JX}(\psim\W\psi) +\psim\W\iota_{\sst JX}\psi\\
					&=&	\iota_{\sst JX}(\psim\W\psi) -\psim\W\iota_{\sst X}\psim\\
					&=& \iota_{\sst JX}(\psim\W\psi) -\psi\W\iota_{\sst X}\psi.
\end{eqnarray*}
Hence, $2\,\iota_{\sst X}\psi\W\psi = \iota_{\sst JX}(\psim\W\psi)$. 
Now, from condition \eqref{normcond} we know that $\psim\W\psi = -\frac23\,\omega^3= -4\,dV_g$. Thus, 
\[
\iota_{\sst X}\psi\W\psi = -2\,\iota_{\sst JX} dV_g =-2*(JX)^\flat = -2*(\iota_{\sst X}\omega).
\]
\end{proof}

\begin{proof}[Proof of Theorem \ref{MainThm}] \ 

Let $X\in\gg$. Then, using the closedness of $\omega$ we have $0=\mathcal{L}_{\sst X}\omega = d(\iota_{\sst X}\omega)$.
Moreover, since $d\psi=0$ and $\mathcal{L}_{\sst X}\psi=0$, then $d(\iota_{\sst X}\psi\W\psi)=0$ and Lemma \ref{prellem} implies that $d*(\iota_{\sst X}\omega)=0$. 
Hence, the 1-form $\iota_{\sst X}\omega$ is $\Delta_g$-harmonic and $\mathscr{F}$ coincides with the injective map $Z\mapsto\iota_{\sst Z}\omega$ 
restricted to $\gg$. From this \ref{thm1}) follows. 

In order to prove \ref{thm2}), we begin recalling that every Killing field on a compact manifold 
preserves every harmonic form. Consequently,  for all $X,Y\in\gg$ we have
\[
0 = \mathcal{L}_{\sst Y}(\iota_{\sst X}\omega) = \iota_{\sst [Y,X]}\omega + \iota_{\sst X}(\mathcal{L}_{\sst Y}\omega) =  \iota_{\sst [Y,X]}\omega.
\]
Since the map $Z\mapsto \iota_{\sst Z}\omega$ is injective, we obtain that $\frg$ is abelian. 
Now, $\G$ is compact abelian and it acts effectively on the compact manifold $M.$ Therefore, the principal isotropy is trivial and $\dim(\gg)\leq6$.  
When $\dim(\gg) = 6$, $M$ can be identified with the 6-torus $\mathbb{T}^6$ endowed with a left-invariant metric, which is automatically flat. 
Hence, if $(\omega,\psi)$ is strict symplectic half-flat, then $\dim(\gg)\leq5$. 

As for \ref{thm3}), we fix a point $p$ of $M$ and we observe that the image of the isotropy representation $\rho:\G_p\rightarrow \mathrm{O}(6)$
is conjugate into $\SU(3)$. Since $\SU(3)$ has rank two and $\G_p$ is abelian, the dimension of $\gg_p$ is at most two. 
If $\dim(\gg_p)=2$, then the image of $\rho$ is conjugate to a maximal torus of $\SU(3)$ and its fixed point set in $T_pM$ is trivial. 
As $T_p(\G\cdot p)\subseteq (T_pM)^{\G_p}$, the orbit $\G\cdot p$ is zero-dimensional, which implies that $\dim(\gg)=2$. 

Assertion \ref{thm4}) is equivalent to proving that $\G_p$ is trivial for every $p\in M$ whenever $\dim(\gg)\geq4$. 
In this case, $\dim(\gg_p)\leq1$ by \ref{thm3}), and therefore $\dim(\G\cdot p)\geq3$. 
If $\G_p$ contains a non-trivial element $h$, then $\rho(h)$ fixes every vector in $T_p(\G\cdot p)$ and, 
consequently, its fixed point set in $T_pM$ must be non-trivial of dimension at least three. 
On the other hand, a non-trivial element of $\SU(3)$ is easily seen to have a fixed point set of dimension at most two. This shows that $\G_p=\{1_{\sst \G}\}$. 
The last assertion follows immediately from \cite{Mos}.
\end{proof}

Point \ref{thm2}) in the above theorem gives a direct proof of a result obtained in \cite{PoRa}. 
\begin{corollary}
There are no compact homogeneous 6-manifolds endowed with an invariant strict symplectic half-flat structure. 
\end{corollary}

It is worth observing here that the non-compact case is less restrictive.  
For instance, it is possible to exhibit non-compact examples which are homogeneous under the action of a {\em semisimple} Lie group of automorphisms 
(see e.g.~\cite{PoRa}). Moreover, in the next section we shall construct non-compact examples of cohomogeneity one with respect to a semisimple Lie group of automorphisms. 

The next example was given in \cite{DeBTom0}. 
It shows that $\G$ can be non-trivial, that the upper bound on its dimension given in \ref{thm2}) can be attained, 
and that \ref{thm4}) is only a sufficient condition. 
\begin{example}\label{ExTorus}
On $\R^6$ with standard coordinates $(x^1,\ldots,x^6)$ consider three smooth functions $a(x^1)$, $b(x^2)$, $c(x^3)$ in such a way that
\[
\lambda_1 \coloneqq b(x^2)-c(x^3),\quad \lambda_2 \coloneqq c(x^3)-a(x^1),\quad \lambda_3 \coloneqq a(x^1)-b(x^2),
\]
are $\Z^6$-periodic. 
Then, the following pair of $\Z^6$-invariant differential forms on $\R^6$ induces an $\SU(3)$-structure on $\T^6 = \R^6/\Z^6$:
\[
\omega 	= dx^{14}+dx^{25}+dx^{36},\quad \psi = -e^{\lambda_3}\,dx^{126} +e^{\lambda_2}\,dx^{135}  -e^{\lambda_1}\,dx^{234} +dx^{456},
\]
where $dx^{ijk\cdots}$ is a shorthand for the wedge product $dx^i\W dx^j \W dx^k \W \cdots$. 
It is immediate to check that $(\omega,\psip)$ is strict symplectic half-flat whenever the functions $\lambda_i$ are not all constant. 
The automorphism group of $(\mathbb{T}^6,\omega,\psip)$ is $\mathbb{T}^3$ when $a(x^1)\,b(x^2)\, c(x^3) \not\equiv0$, 
while it becomes $\mathbb{T}^4$  ($\mathbb{T}^5$) when one (two) of them vanishes identically.  
\end{example}

Finally, we observe that there exist examples where the upper bound on the dimension of $\gg$ given in \ref{thm1}) is more restrictive than the upper bound given in \ref{thm2}). 
\begin{example}
In \cite{CoTo}, the authors obtained the classification of six-dimensional nilpotent Lie algebras admitting symplectic half-flat structures. 
The only two non-abelian cases are described up to isomorphism by the following structure equations
\[
(0,0,0,0,e^{12},e^{13}),\quad (0,0,0,e^{12},e^{13},e^{23}). 
\]

Denote by $\mathrm{N}$ the simply connected nilpotent Lie group corresponding to one of the above Lie algebras, and endow it with a left-invariant 
strict symplectic half-flat structure $(\omega,\psi)$.   
By \cite{Mal}, there exists a co-compact discrete subgroup $\Gamma\subset \mathrm{N}$ giving rise to a compact nilmanifold $\Gamma\backslash \mathrm{N}$.     
Moreover, the left-invariant pair $(\omega,\psi)$ on $\mathrm{N}$ passes to the quotient defining an $\SU(3)$-structure of the same type on $\Gamma\backslash \mathrm{N}$.  
By \cite{Nom}, we have that $b_1(\Gamma\backslash \mathrm{N})$ is either four or three.
\end{example}

\section{Non-compact cohomogeneity one examples}\label{sect3}
In this section, we construct complete examples of strict symplectic half-flat structures on a non-compact 6-manifold admitting a cohomogeneity one action of a 
semisimple Lie group of automorphisms.  
This points out the difference between the compact and the non-compact case, and together with the results in \cite[$\S$4.3]{PoRa} it suggests that the non-compact ambient provides 
a natural setting to obtain new examples.   

From now on, we consider the natural cohomogeneity one action on $M= T\mathbb{S}^3\cong\mathbb{S}^3\times\R^3$ 
induced by the transitive $\SO(4)$-action on $\mathbb{S}^3$. Then, we have
\[
T\mathbb{S}^3 \cong \SO(4)\times_{\sst \SO(3)}\R^3. 
\]

We refer the reader to \cite{AlAl,Mos,PoSp1,PoSp} for basic notions on cohomogeneity one isometric actions.  
Following the notation of \cite{PoSp}, we consider the Lie algebra $\frso(4) \cong \frsu(2)+\frsu(2)$ and we fix the following basis of $\frsu(2)$
\[
H	\coloneqq	 \frac{1}{2}\, \left(\begin{array}{cc}i& 0 \\ 0 & -i  \end{array}\right), 		\quad 
E	\coloneqq	\frac{1}{2\sqrt{2}}\, \left(\begin{array}{cc} 0 & 1\\  -1 & 0 \end{array}\right), 	\quad 
V	\coloneqq	\frac{1}{2\sqrt{2}}\,  \left(\begin{array}{cc}  0 & i \\ i&0\end{array}\right). 
\]

Let $\gamma:\R\rightarrow M$ be a normal geodesic such that $p\coloneqq\gamma(0)\in\mathbb{S}^3$ and $\gamma_t \coloneqq\gamma(t)$ is a regular point for all $t\neq0$. 
The singular isotropy subalgebra is $\frso(4)_p=\frsu(2)_{\sst \mathrm{diag}}$, while the principal isotropy subalgebra $\frk\coloneqq\frso(4)_{\sst \gamma_t}$, 
$t\neq0$, is one-dimensional and spanned by $(H,H).$ 
We consider the following basis of $\frso(4)\cong\frsu(2)+\frsu(2)$
\[
\renewcommand\arraystretch{1.4}
\begin{array}{crlc}
E_1 \coloneqq (E,0),	&	V_1 \coloneqq (V, 0),	& 	E_2 \coloneqq (0, E),	&	V_2 \coloneqq (0,V),\\
				& 	U \coloneqq (H, H), 	&	A \coloneqq (H, - H).		&
\end{array}
\renewcommand\arraystretch{1}
\]
We let $\xi\coloneqq \frac{\partial}{\partial t}$, and for any $Z\in\frso(4)$ we denote by $\widehat{Z}$ the corresponding fundamental vector field on $M.$ 
Then, a basis of $T_{\sst \gamma_t}M$ for $t\neq0$ is given by
\[
(\xi, \widehat A, \widehat E_1, \widehat V_1,  \widehat E_2, \widehat V_2)_{\sst\gamma_t}. 
\]
We shall denote the dual coframe along $\gamma_t$ by $(\xi^*, A^*, E_1^*, V^*_1, E^*_2, V^*_2)_{\sst\gamma_t}$, where $\xi^*\coloneqq dt$. 

Let $\K\subset\SO(4)$ be the principal isotropy subgroup corresponding to the Lie algebra $\frk$. 
The space of $\K$-invariant 2-forms on $T_{\sst\gamma_t}M$, $t\neq0,$ is spanned by 
\[
\omega_1 \coloneqq \xi^*\wedge A^*,\qquad \omega_2 \coloneqq E_1^*\wedge V_1^*,\qquad\ \omega_3 \coloneqq E_2^*\wedge V_2^*,
\]
\[
\omega_4 \coloneqq E_1^*\wedge E_2^* + V_1^*\wedge V_2^*,\qquad \omega_5 \coloneqq E_1^*\wedge V_2^*- V_1^*\wedge E_2^*.
\]
These forms extend as $\SO(4)$-invariant $2$-forms on the regular part $M_{\sst0}\coloneqq \mathbb{S}^3\times\R^{\sst+}.$ 
By \cite{PoSp}, their differentials along $\gamma_t$ are 
\begin{equation}\label{differentials} 
\renewcommand\arraystretch{1.4}
\begin{array}{lr}
d {\o_1}|_{\g_t} = \frac{1}{4}\,\xi^* \wedge  \left(\o_2 - \o_3\right),		&	d{\o_2}|_{\g_t} = d{\o_3}|_{\g_t} = 0\ , \\
d{\o_4}|_{\g_t} =  -2\, A^*\wedge \omega_5,					&	d{\omega_5}|_{\g_t} = 2\, A^*\wedge \omega_4.
\end{array}
\renewcommand\arraystretch{1}
\end{equation}

We now describe the general $\SO(4)$-invariant symplectic 2-form $\omega$ on $M.$ 
Along $\gamma_t$, $t\neq0$, we have
\[
\omega|_{\gamma_t} = \sum_{i=1}^5 f_i(t)\,\omega_i,
\]
for suitable smooth functions $f_i\in\mathcal{C}^\infty(\R^{\sst+})$. 
By \cite[Prop.~6.1]{PoSp}, the $\SO(4)$-invariant 2-form $\omega$ on $M_{\sst0}$ corresponding to $\omega|_{\gamma_t}$ admits a smooth extension to the whole $M$ if and only if 
the functions $f_i$ extend smoothly on $\R$ as follows:
\begin{enumerate}[i)]
\item\label{i} $f_1$ and $f_4$ are even and $f_2$, $f_3$, $f_5$ are odd;
\item\label{ii} $f_3'(0)=\frac12\,f_1(0)+f_2'(0)$, $f_5'(0)=-\frac14\,f_1(0)-f_2'(0)$, and $f_4(0)=0$.
\end{enumerate} 
Moreover, $\omega|_p$ is non-degenerate if and only if $f_1(0)\neq0$. 

Using \eqref{differentials}, we compute $d\omega$ and we see that $\omega$ is closed if and only if 
\[
f_4,f_5\equiv0,\quad 
\begin{cases}
f_2' = -\frac14\,f_1\\
f_3' = \frac14\,f_1
\end{cases}.
\]
Combining this with the extendability conditions, we obtain that every $\SO(4)$-invariant symplectic 2-form $\omega$ on $M$ can be written as
\begin{equation} \label{omega}
\omega|_{\sst\gamma_t} = f_1(t)\,\omega_1+ f_2(t)\,\omega_2+ f_3(t)\,\omega_3,\quad t\neq0,
\end{equation}
with $f_1\in\mathcal{C}^\infty(\R)$ even and nowhere vanishing, and 
\[
f_2(t) = -\frac14\int_0^t f_1(s)\,ds = -f_3(t). 
\]
Notice that $\omega^3|_{\gamma_t} = -6f_1f_2^2\, \o_1\W\o_2\W\o_3$ at every regular point of the geodesic $\gamma_t$. 
As $f_1$ is nowhere zero, we may assume that $f_1<0$, so that the volume form $\xi^* \W A^* \W E_1^* \W V^*_1 \W E^*_2 \W V^*_2$  
defines the same orientation on $T_{\sst\gamma_t}M$ as $\frac{1}{6}\,\omega^3|_{\sst\gamma_t}$ for all $t\in\R^{\sst+}$.  

We now search for an $\SO(4)$-invariant closed 3-form $\psi\in \Omega^3(M)^{\sst\SO(4)}$ so that the pair $(\omega,\psi)$ defines an $\SO(4)$-invariant symplectic half-flat 
structure on $M.$ For the sake of simplicity, we make the following Ansatz
\[
\psi=du,~u \in \Omega^2(M)^{\sst\SO(4)}.
\]
As before, along $\gamma_t$, $t\neq0$, we can write
\begin{equation}
\label{u}
u|_{\gamma_t} = \sum_{i=1}^5 \phi_i(t)\,\omega_i, 
\end{equation}
for some smooth functions $\phi_i\in\mathcal{C}^\infty(\R^{\sst+})$ satisfying the same extendability conditions as the $f_i$'s.  
Then, omitting the dependence on $t$ for brevity, we have
\begin{equation}\label{psiplus}
\psi|_{\gamma_t} = \psi_2\,\xi^*\W\omega_2 + \psi_3\,\xi^*\W\omega_3 +\phi_4'\,\xi^*\W\omega_4 + \phi_5'\,\xi^*\W\omega_5 +2\,A^*\W(\phi_5\,\omega_4-\phi_4\,\omega_5), 
\end{equation}
where $\psi_2\coloneqq \frac14\,\phi_1+\phi_2'$ and $\psi_3\coloneqq \phi_3'-\frac14\,\phi_1. $

By \cite{Hit}, the pair $(\omega,\psi)$ defines an $\SU(3)$-structure if and only if the following conditions hold:
\begin{enumerate}[a)]
\item\label{a} the compatibility condition $\omega\W\psi=0$;
\item\label{b} the stability condition $P(\psi)<0$, $P$ being the characteristic quartic polynomial defined on 3-forms (see below for the definition);
\item\label{c} denoted by $J$ the almost complex structure induced by $(\omega,\psi)$, then the complex volume form $\Psi\coloneqq\psi+i\,\psim$ 
with $\psim\coloneqq J\psi$ fulfills the normalization condition \eqref{normcond}; 
\item\label{d} the symmetric bilinear form $g\coloneqq\omega(\cdot,J\cdot)$ is positive definite. 
\end{enumerate} 

The compatibility condition \ref{a}) along $\gamma_t$ reads $f_2 \psi_3+f_3 \psi_2=0$. 
Since $f_2=-f_3\neq0$, this implies 
\begin{equation}\label{psi2}
\psi_2=\psi_3. 
\end{equation}

Recall that at each point $q\in M$ the 3-form $\psi$ gives rise to an endomorphism $S\in\End(T_qM)$ defined as follows for every $\theta\in T_q^*M$ and every $v\in T_qM$
\[
\iota_{\sst v}\psi\W\psi\W\theta= \theta(S(v))\frac{\omega^3}{6}. 
\]
The endomorphism $S$ satisfies $S^2=P(\psi)\mathrm{Id}$, and it gives rise to the almost complex structure $J\coloneqq \frac{1}{\sqrt{|P(\psi)|}}S$ when $P(\psi)<0$. 

From the expressions
\begin{eqnarray*}
\iota_{\sst\xi}\psi\wedge\psi|_{\sst\gamma_t} &=& 2 \left(\psi_2^2-(\phi_4')^2-(\phi_5')^2\right)\,\xi^*\wedge\o_2\wedge\o_3-4\left(\phi_4'\phi_5 - \phi_4\phi_5'\right)\,A^*\W\o_2\W\o_3,\\
\iota_{\sst\widehat A}\psi\wedge\psi|_{\sst\gamma_t} &=& 4\left(\phi_4\phi_5'-\phi_4'\phi_5\right)\,\xi^*\wedge\o_2\wedge\o_3-8\left(\phi_4^2+\phi_5^2\right)\,A^*\wedge\o_2\wedge\o_3, 
\end{eqnarray*}
we see that the endomorphism $S\in\End(T_{\sst\gamma_t}M)$ maps the subspace of $T_{\sst\gamma_t}M$ spanned by $\xi$ and $\widehat A|_{\sst\gamma_t}$ 
into itself with associated matrix given by 
\begin{equation}\label{S} -\frac{1}{f_1f_2^2} 
\left(
\renewcommand\arraystretch{1.4}
\begin{array}{cc} 
4\left(\phi_4'\phi_5 - \phi_4\phi_5'\right)			&	8 \left(\phi_4^2+\phi_5^2\right)\\ 
2\left(\psi_2^2 - (\phi_4')^2 - (\phi_5')^2\right)		&	- 4\left(\phi_4'\phi_5 - \phi_4\phi_5'\right) 
\end{array}
\renewcommand\arraystretch{1.4}
\right).
\end{equation}

Since the curve $\gamma_t$ must be a normal geodesic for the metric $g$ 
induced by $(\omega,\psi)$, it follows that the tangent vector $\xi$ is orthogonal to the orbit $\SO(4)\cdot\gamma_t$ at every regular point of $\gamma_t$. In particular, we have
\[
0=g(\xi,\widehat A) = \o(\xi,J(\widehat{A})) = \frac{1}{\sqrt{|P(\psi)|}}\,\o(\xi,S(\widehat{A})) = \frac{4}{f_2^2\,\sqrt{|P(\psi)|}}\,(\phi_4'\phi_5 - \phi_4\phi_5'),
\] 
from which we get
\begin{equation}\label{phi45}
\phi_4'\phi_5 = \phi_4\phi_5'.
\end{equation}
Using \eqref{psi2}, \eqref{phi45} and the definition of $P(\psi)$, we obtain 
\begin{equation}\label{P}
P(\psi) = \frac{16}{f_1^2f_2^4} \left(\phi_4^2+\phi_5^2\right)  \left(\psi_2^2 - (\phi_4')^2 - (\phi_5')^2\right). 
\end{equation}
Consequently, the stability condition \ref{b}) reads
\begin{equation}\label{stabilitycond}
\psi_2^2 - (\phi_4')^2 - (\phi_5')^2<0,\quad \phi_4^2+\phi_5^2 \neq 0,
\end{equation}
for all $t\in\R^{\sst+}$.

We now note that the vector field $J(\xi)$ is tangent to the $\SO(4)$-orbits and it belongs to the space of $\K$-fixed vectors in $T_{\sst\gamma_t}(\SO(4)\cdot\gamma_t)^\K$, 
which is spanned by $\widehat{A}|_{\gamma_t}$. 
Since the geodesic $\g_t$ has unit speed, we see that 
\begin{equation}\label{g11}
1 = g(\xi,\xi) = \o(\xi,J(\xi)) = -\frac{2}{f_2^2\sqrt{|P(\psi)|}} \left(\psi_2^2 - (\phi_4')^2 - (\phi_5')^2\right).
\end{equation}
Using \eqref{P}, the relation \eqref{g11} implies that 
\begin{equation} \label{Pg11}
4\left(\phi_4^2 + \phi_5^2\right) = f_1^2\left((\phi_4')^2 + (\phi_5')^2 - \psi_2^2\right).
\end{equation}     

Let us now focus on \ref{c}).  
From \eqref{S} and \eqref{Pg11}, we obtain $J(\xi) = \frac{1}{f_1}\hat A|_{\gamma_t}$. 
Using this and the identity $\psim = J\psi = -\psi(J\cdot,\cdot,\cdot)$, we have
\begin{equation}\label{hatpsi}
\psim|_{\gamma_t} = \xi^*\W\left(2\, \frac{\phi_4}{f_1}\, \o_5 - 2\, \frac{\phi_5}{f_1}\, \o_4\right) + f_1\, A^*\W\left(\psi_2\, (\o_2+ \o_3) + \phi_4'\, \o_4 + \phi_5'\, \o_5\right).
\end{equation} 
Now, the normalization condition $\psi\W\psim =\frac23\,\omega^3$ gives
\[
4(\phi_4^2+\phi_5^2)- f_1^2\ (\psi_2^2 - (\phi_4')^2 - (\phi_5')^2) = 2\,f_1^2f_2^2.
\]
Combining this with \eqref{Pg11}, we obtain 
\begin{equation}\label{normeq}
\phi_4^2+\phi_5^2 = \frac 14  (f_1 f_2)^2.
\end{equation} 
 Note that \eqref{P}, \eqref{Pg11} and \eqref{normeq} imply  
\[
P(\psi) \equiv -4
\]
along the geodesic $\gamma_t$. Thus, the stability of $\psi$ holds also at $t=0$. 

Going back to \eqref{phi45}, we see that either $\phi_4 = \lambda\phi_5$ or $\phi_5 = \l\phi_4$ for some $\l\in \mathbb R\smallsetminus\{0\}$. 
Since $\phi_4$ and $\phi_5$ extend as an even and an odd function on $\mathbb{R}$, respectively, we see that either $\phi_4\equiv 0$ or $\phi_5\equiv 0$. 
As $f_1f_2$ is an odd function on $\mathbb R$, \eqref{normeq} implies that 
\begin{equation} \label{phi5}
\phi_4\equiv 0,\quad \phi_5 = \pm\frac 12 f_1f_2.
\end{equation}  

The matrix associated with the symmetric bilinear form $\omega(\cdot,J\cdot)$ along $\gamma_t$, $t\in \mathbb R^+,$ is 
\[
\left( 
\renewcommand\arraystretch{1.3}
\begin{array}{cccccc} 
	1	& 0		& 0			& 0	& 0	& 0\\ \noalign{\medskip}
	0	& f_1^2	& 0			& 0	& 0	& 0\\ \noalign{\medskip}
	0	& 0		& -2\, \frac {\phi'_5 \phi_5}{f_1f_2}	& 0	& -2\, \frac { \psi_2 \phi_5}{f_1 f_2}	& 0\\ \noalign{\medskip}
	0	& 0		& 0			& -2\,\frac {\phi'_5 \phi_5}{f_1 f_2}	& 0	& -2\,\frac {\psi_2 \phi_5}{f_1 f_2}\\ \noalign{\medskip}
	0	& 0		& -2\, \frac {\psi_2 \phi_5}{f_1 f_2}	& 0	& -2\, \frac{\phi'_5 \phi_5}{f_1 f_2}	& 0\\ \noalign{\medskip}
	0	& 0		& 0			& -2\, \frac { \psi_2  \phi_5}{f_1 f_2}	& 0	& -2\, \frac {\phi'_5 \phi_5}{f_1 f_2}
\end{array}
\renewcommand\arraystretch{1}
\right), 
\]
and condition \ref{d}) can be written as 
\[
-2\, \frac {\phi'_5 \phi_5}{f_1f_2}> 0,\quad \psi_2^2 < (\phi_5')^2.
\] 
The former condition is equivalent to $(f_2^2)''>0$, while the latter is satisfied whenever $\psi$ is stable (cf.~\eqref{stabilitycond}).

Note that the metric $g$ extends smoothly over the singular orbit $\mathbb{S}^3$ to a Hermitian symmetric bilinear form. 
The restriction of $g$ on $T_p\mathbb{S}^3$ is positive definite as $g_p(\widehat{A},\widehat{A}) = f_1^2(0) > 0$ and the orbit $\SO(4)\cdot p$ is isotropy irreducible. 
Moreover, $T_pM=T_p\mathbb{S}^3\oplus J(T_p\mathbb{S}^3)$, and from this we see that $g_p$ is positive definite.

Summing up, the existence of a complete $\SO(4)$-invariant symplectic half-flat structure $(\omega,\psi)$ on $M$ is equivalent to the existence of 
a smooth function $f_1\in \mathcal{C}^\infty(\R)$ satisfying the following conditions: 
\begin{enumerate}[1)]
\item\label{1} $f_1$ is even and negative;
\item\label{2} the function $f_2(t)\coloneqq -\frac 14\int_0^t f_1(s) ds$ satisfies $(f_2^2)'' > 0$; 
\item\label{3} there exists an even smooth function $\psi_2\in\mathcal{C}^\infty(\R)$  satisfying $\psi_2^2 = [(f_2^2)'']^2 - f_2^2$.
\end{enumerate}
Indeed, given $f_1$ we define the symplectic form $\o$ on $M$ as in \eqref{omega}, with $f_3=-f_2$. 
As for $\psi$, we let $\psi_3 \coloneqq \psi_2$, $\phi_4\coloneqq0$, and $\phi_5 \coloneqq \pm\frac 12 f_1f_2$ in \eqref{psiplus}. 
Then, \eqref{Pg11} and \eqref{normeq} imply $\psi_2^2 = (\phi_5')^2 - f_2^2$, and we can choose the sign in the definition of $\phi_5$ so that the extendability condition 
$\phi_5'(0) = -\psi_2(0)$ given in \ref{ii}) is satisfied. 
It is also easy to see that we may choose $\phi_1,\phi_2,\phi_3$ so that $\psi_2= \frac14\,\phi_1+\phi_2'$ and $\psi_3= \phi_3'-\frac14\,\phi_1$,  
and the corresponding $u$ as in \eqref{u} extends to a global $2$-form on $M.$ 
The resulting $3$-form $\psi$ is then stable by condition \ref{3}) and \eqref{P}. 
The stability condition together with the inequality in \ref{2}) implies that the induced bilinear form $g$ is everywhere positive definite. 
Hence, we have proved the following result. 
\begin{proposition}
The existence of a complete $\SO(4)$-invariant symplectic half-flat structure $(\omega,\psi)$ on $T\mathbb{S}^3 = \SO(4)\times_{\sst\SO(3)}\R^3$ 
with $\psi \in d\Omega^2(M)$ is equivalent to the existence of 
a smooth function $f_1\in\mathcal{C}^\infty(\R)$ satisfying conditions \rm{\ref{1}), \ref{2}), \ref{3})}. 
\end{proposition}

Recall that the symplectic half-flat structure $(\omega,\psi)$ is strict if and only if 
the unique 2-form $\sigma\in\left[\Omega^{\sst1,1}_{\sst0}(M) \right]$ fulfilling $d\psim=\sigma\W\omega$ is not identically zero. 
Starting from \eqref{hatpsi}, using \eqref{differentials} and the identity $dA^*|_{\sst\gamma_t} = \frac14\left(\omega_3-\omega_2\right)$ 
(cf.~\cite[(3.27)]{PoSp}), we obtain
\[
d\psim|_{\sst \gamma_t} = \left(\left(f_1 \phi_5'\right)'-4\,\frac{\phi_5}{f_1} \right)\,\omega_1\W\omega_5 + \left(f_1 \psi_2\right)'\,\omega_1\W(\omega_2+\omega_3), 
\]
whence 
\[
\sigma|_{\sst\gamma_t} = \frac{1}{f_1}\left(f_1\psi_2\right)'(\omega_2+\omega_3) + \frac{1}{f_1}\left(\left(f_1 \phi_5'\right)'-4\,\frac{\phi_5}{f_1} \right) \omega_5.
\]  

By \cite{BedVez}, we know that the scalar curvature of the metric $g$ induced by a symplectic half-flat structure is given by $\mathrm{Scal}(g) = -\frac12|\sigma|^2$. 
Hence, in our case we have
\[
\mathrm{Scal}(g)|_{\sst\gamma_t}	= -\frac{1}{f_1^2f_2^2}\left[ \left((f_1\psi_2)'\right)^2 - \left(\left(f_1 \phi_5'\right)'-4\,\frac{\phi_5}{f_1} \right)^2 \right]
							= -\left( \frac{\left(f_1\psi_2\right)'}{f_1\phi_5'} \right)^2,
\]
where the second equality follows from the relations obtained so far. 

We may construct plenty of complete $\SO(4)$-invariant strict symplectic half-flat structures on $M$ by choosing a suitable $f_1$ as above. For instance, the function 
\[
f_1(t) \coloneqq  -\cosh(t),\quad t\in\R,
\] 
fits in with conditions \ref{1}), \ref{2}), \ref{3}). With this choice, the scalar curvature is
\[
\mathrm{Scal}(g)|_{\sst\gamma_t} = -\tanh^2(t) \frac{   ( 6 \cosh^2(t) - 5 )^2 }{4 \cosh^4(t) -8\cosh^2(t) +5   }.
\]
This shows that the resulting symplectic half-flat structure is strict and non-homogeneous. 

Note that the vanishing of $\sigma$ is equivalent to the vanishing of $\mathrm{Scal}(g)$. 
Hence, setting $(f_1\psi_2)'=0$, the resulting $\SU(3)$-structure $(\omega,\psi)$ is Calabi-Yau and the associated metric is the well-known Stenzel's Ricci-flat metric 
on $T\mathbb{S}^3$ (cf.~\cite{Ste}). 

Finally, we remark that the scalar curvature always vanishes at $t=0$. Indeed, $(f_1\psi_2)'(0)=0$, as $f_1\psi_2$ is even, while $f_1(0)\phi_5'(0) \neq0$. 
This implies that an $\SO(4)$-invariant symplectic half-flat structure $(\omega,\psi)$ with $\psi$ exact has constant scalar curvature if and only if it is Calabi-Yau.

\end{document}